\documentclass[12pt]{article}
\usepackage{}
\usepackage[numbers,sort&compress]{natbib}

\usepackage{color}

\usepackage[english]{babel}
\usepackage{t1enc}
\usepackage{amsthm,amsmath,amssymb}
\usepackage{mathrsfs}
\usepackage[mathscr]{eucal}
\usepackage[latin1]{inputenc}
\usepackage[english]{babel}
\usepackage{latexsym}
\usepackage{graphicx}
\usepackage[noend]{algpseudocode}
\usepackage{algorithmicx,algorithm}
\usepackage{lineno}
\usepackage{booktabs}
\usepackage{multirow}
\newtheorem{theorem}{Theorem}
\newtheorem{corollary}{Corollary}
\newtheorem{proposition}{Proposition}
\newtheorem{lemma}{Lemma}
\newtheorem{remark}{Remark}

\newtheorem{conjecture}{Conjecture}

\newtheorem{example}{Example}
\newtheorem{definition}[theorem]{Definition}

\def \T{\textup{T}}

\def \rank{\textup{rank}}
\def \sgn{\textup{sgn}}
\def \tr{\textup{tr}}
\def \diag{\textup{diag~}}

\makeatletter
\newcommand{\rmnum}[1]{\romannumeral #1}
\newcommand{\Rmnum}[1]{\expandafter\@slowromancap\romannumeral #1@}
 
\makeatother
\title{Graphs with at most one generalized cospectral mate}
\author{\small Wei Wang$^{{\rm a}}$\quad\quad Wei Wang$^{\rm b}$\thanks{Corresponding author: wang\_weiw@163.com}\quad\quad Tao Yu$^{{\rm c}}$
\\
{\footnotesize$^{\rm a}$School of Mathematics, Physics and Finance, Anhui Polytechnic University, Wuhu 241000, P. R. China}\\
{\footnotesize$^{\rm b}$School of Mathematics and Statistics, Xi'an Jiaotong University, Xi'an 710049, P. R. China}\\
{\footnotesize$^{\rm c}$School of Mathematics, 686 Cherry Street, Georgia Institute of Technology, Atlanta, GA 30332, USA}
}
\date{}
\begin{document}
 \maketitle
\begin{abstract}
	Let $G$ be an $n$-vertex graph with adjacency matrix $A$, and $W=[e,Ae,\ldots,A^{n-1}e]$ be the walk matrix of $G$, where $e$ is the all-one vector. In	Wang [J. Combin. Theory, Ser. B, 122 (2017): 438-451], the author showed that any graph $G$ is uniquely determined by its generalized spectrum (DGS) whenever $2^{-\lfloor n/2 \rfloor}\det W$ is odd and square-free. In this paper, we introduce a large family of graphs
	$$\mathcal{F}_n=\{\text{$n$-vertex graphs~} G\colon\, 2^{-\lfloor n/2 \rfloor}\det W =p^2b\text{~and~} \rank~W=n-1\text{~over~} \mathbb{Z}/p\mathbb{Z}\},$$
	where $b$ is odd and square-free, $p$ is an odd prime and $p\nmid b$. We prove that any graph in $\mathcal{F}_n$ either is DGS or has exactly one generalized cospectral mate up to isomorphism. Moreover, we show that the problem of finding the generalized cospectral mate for a graph in $\mathcal{F}_n$ is equivalent to that of generating an appropriate rational orthogonal matrix from a given integral vector. This equivalence essentially depends on an amazing property of graphs in terms of generalized spectra, which states that any symmetric integral matrix generalized cospectral with the adjacency matrix of some graph  must be an adjacency matrix. Based on this equivalence, we develop an efficient algorithm to decide whether a given graph in $\mathcal{F}_n$ is DGS and further to find the unique  generalized cospectral mate when it is not.  We give some
	experimental results on graphs with at most 20 vertices, which suggest that $\mathcal{F}_n$ may have a positive density (nearly $3\%$) and possibly almost all graphs in $\mathcal{F}_n$ are DGS as $n\rightarrow \infty$.  This gives a supporting evidence for Haemers' conjecture that almost all graphs are determined by their spectra. \\


\end{abstract}
\section{Introduction}
\label{intro}
All graphs considered in this paper are simple and undirected. The \emph{spectrum} of a graph $G$, denoted by $\sigma(G)$, is the multiset of all eigenvalues of
its adjacency matrix. The \emph{generalized spectrum} of a graph $G$ is defined to be the pair $(\sigma(G), \sigma(\bar{G}))$, where
$\emph{G}$ is the complement of $G$. Two graphs are \emph{generalized cospectral} if they share the same generalized spectrum. Clearly, isomorphic graphs are generalized cospectral, but the converse is not true in general.  Two  graphs $G$ and $H$  are called a pair of \emph{generalized cospectral  mates} if they are generalized cospectral but nonisomorphic.  A graph $G$ is \emph{determined by generalized spectrum} (or DGS for short) if it has no generalized cospectral mates, that is, all graphs having the same generalized spectrum as $G$ are isomorphic to $G$.  We remark that in the context of classical adjacency spectrum, the  corresponding notions have received considerable attention. We refer the readers to \cite{ervdamLAA2003,ervdamDM2009}.

We are mainly concerned with the generalized spectra of graphs in this paper. For a given graph $G$, a natural problem is to determine whether $G$ is DGS or not, or more subtly, to find some or all (if any) generalized cospectral  mates of $G$.  The problem turns out to be very difficult in general. Nevertheless, Wang \cite{wang2013ElJC,wang2017JCTB} found a strong connection between this problem and the properties of walk matrices of graphs. For a graph $G$ with $n$ vertices, the \emph{walk matrix} of $G$, denoted by $W(G)$ or simply $W$, is the $n\times n$ matrix $[e,Ae,\ldots,A^{n-1}e]$, where $A$ is the adjacency matrix of $G$ and $e$ is the all-one column  vector of dimension $n$. The following simple arithmetic condition on $\det W$ for a graph $G$ being DGS was obtained in \cite{wang2017JCTB}.
\begin{theorem}\cite{wang2017JCTB}\label{osf}
	If $2^{-\lfloor \frac{n}{2}\rfloor}\det W$ (which is always an integer) is odd and square-free, then $G$ is DGS.
\end{theorem}
The condition of Theorem \ref{osf} is the best possible in the sense that if $\det W$ has a multiple odd prime factor then $G$ may not be DGS. A small counterexample can be found in \cite{wang2013ElJC}.  The general idea hidden in that counterexample was revealed by the following theorem. For an integral matrix $M$ and a prime $p$, we use $\rank_p M$ to denote the rank of $M$  over the finite filed $\mathbb{F}_p=\mathbb{Z}/p\mathbb{Z}$.
\begin{theorem}\cite{wang2019LAA}\label{GGM}
	Let $p$ be an odd prime. Suppose that $\det W\neq 0$ and the following conditions hold:\\
\noindent (\rmnum{1}) $p^2\mid \det W$;\\
\noindent (\rmnum{2}) $\rank_p W=n-1$;\\
\noindent (\rmnum{3}) $W^\T z\equiv 0\pmod{p}$ has a solution (permuting the entries, or equivalently reordering the vertices, if necessary) of the form
	$$(\underbrace{-1,-1,\ldots,-1}_{p},\underbrace{1,1,\ldots,1}_{p},0,0,\ldots,0).$$
Then $G$ is not DGS. Furthermore $G$ has a generalized cospectral mate whose adjacency matrix is similar to $A$ via a rational orthogonal matrix
\begin{equation}
Q=\left(\begin{matrix}
\frac{1}{p}\left(\begin{matrix}pI_p-J&J\\J&pI_p-J\end{matrix}\right)&&\\&&I_{n-2p}
\end{matrix}\right),
\end{equation}  where $I$ and $J$ are the identity matrix and the all-one matrix, respectively.
\end{theorem}
The key point of Theorem \ref{GGM} is that, under the stated assumptions,  the matrix $Q^\T A Q$ must be a symmetric $(0,1)$-matrix with vanishing diagonal, that is, an adjacency matrix. We remark that the established pair of generalized cospectral mates also has a clear meaning from the viewpoint of edge switchings. This kind of switching is referred to as generalized GM-switching, which, as an analogue to the original GM-switching method introduced in \cite{godsil1982},  can be used to construct some new pairs of generalized cospectral mates; see \cite{ihringer2019LAA, ihringer2021DM} for some recent application of the generalized GM-switching in constructing cospectral strongly regular graphs.

The main weakness of the above theorem is the third condition. The required solution seems so special that it can rarely be satisfied. A natural question is that whether there exist some other kinds of solutions to  guarantee the existence of a generalized cospectral mate for $G$. What is the exact relationship between the DGS-property  of $G$ and the solutions to $W^\T z\equiv 0\pmod{p}$?

In this paper, we shall introduce a large family of graphs closely related to the first two conditions of Theorem \ref{GGM}.  The main discovery is that for this family of graphs,  the DGS-property of a graph can be completely determined from any nontrivial solution to the equation $W^\T z\equiv 0\pmod{p}$. To give the definition, we first recall some basic fact on Smith normal form of an integral matrix.

Two $n\times n$ integral matrices $M_1$ and $M_2$ are \emph{integrally equivalent} if $M_2$ can be
obtained from $M_1$ by a sequence of the following operations: row permutation, row negation, addition of an integer multiple of one row to another and the corresponding column operations. Any integral invertible matrix $M$ is integrally equivalent to a diagonal matrix $\diag[d_1, d_2,..., d_n]$, known as the \emph{Smith normal form} of $M$, in which  $d_1,d_2,\ldots,d_r$ are positive integers with $d_i\mid d_{i+1}$ for $i = 1,...,n-1$. We are mainly interested in the Smith normal form of an invertible walk matrix. A particular interesting example is the walk matrix for graphs satisfying the condition of Theorem \ref{osf}.

\begin{theorem}\cite{wang2017JCTB}\label{maxrank}
If
	$2^{-\lfloor\frac{n}{2}\rfloor}\det W= b$ for some odd and square-free integer $b$, then  the Smith normal form of $W$ is
	$$\diag[\underbrace{1,1,\ldots,1}_{\lceil\frac{n}{2}\rceil},\underbrace{2,2,\ldots,2,2b}_{\lfloor\frac{n}{2}\rfloor}].$$
\end{theorem}
Now we introduce a  family of graphs using  smith normal forms of walk matrices.
\begin{definition}\label{deff}\normalfont
For a positive integer $n$, we use	$\mathcal{F}_n$  to denote the family of all graphs $G$ of order $n$ such that the Smith normal form of $W(G)$ is 	$$\diag[\underbrace{1,1,\ldots,1}_{\lceil\frac{n}{2}\rceil},\underbrace{2,2,\ldots,2,2p^2b}_{\lfloor\frac{n}{2}\rfloor}],$$
where $b$ is odd and square-free, $p$ is an odd prime and $p\nmid b$.
\end{definition}
\begin{remark}\normalfont
The unique odd prime $p$ satisfying $p^2\mid \det W$ is a crucial parameter for a graph $G\in \mathcal{F}_n$. We shall use the notation $  \mathcal{F}_n^p=\{G\in \mathcal{F}_n\colon\, p^2\mid \det W\}$.
\end{remark}
We note that graphs in $\mathcal{F}_n$ can also be equivalently defined as graphs satisfying (\rmnum{1}) $2^{-\lfloor n/2 \rfloor}\det W =p^2b$ and  (\rmnum{2}) $\rank_p W=n-1$ simultaneously, with the same assumptions on $b$ and $p$ as in Definition \ref{deff}. In particular, every graph in $\mathcal{F}_n$  clearly satisfies the first two condition of Theorem \ref{GGM}. Compared the Smith normal form in Definition \ref{deff} with that in Theorem \ref{maxrank}, the only difference is the last invariant factor. For graphs in $\mathcal{F}_n$, the last invariant contains exactly one square factor. Intuitively, since a graph in $\mathcal{F}_n$ \emph{almost} satisfies the condition of Theorem \ref{osf}, it may be \emph{almost} determined by its generalized spectrum. Indeed, we shall prove the following theorem.
\begin{theorem}\label{atmost1}
	Every graph in $\mathcal{F}_n$ has at most one generalized cospectral mate.
\end{theorem}

We shall prove Theorem \ref{atmost1} in Section 2.2. We relate any possible generalized cospectral mates of $G\in \mathcal{F}_n^p$ to a particular kind of orthogonal matrices, which we call \emph{primitive} matrices. We show that for a fixed graph $G\in \mathcal{F}_n^p$ not being DGS, all possible  primitive matrices related to  $G$ are  unique up to column permutations. In Section 2.3, we further establish the equivalence between the existence of a generalized cospectral mate for a graph and the existence of a primitive matrix.

In order to give a complete criterion to distinguish two different kinds (DGS v.s. non-DGS) of graphs in $\mathcal{F}_n$, in Section 3 we develop a procedure to generate all possible primitive matrices from a given vector. When it succeeds, it finds a generalized cospectral mate; when it fails, it indicates that the given graph is DGS. Using the proposed algorithm, we conduct a numerical experiment on graphs with at most 20 vertices,  which suggests that, for not too small $n$, while $\mathcal{F}_n$ may have a stable positive density (nearly 3\%), almost none of $\mathcal{F}_n$ has a generalized cospectral mate. This gives some evidences for Haemers' conjecture that almost all graphs are determined by their spectra.

\section{Existence and uniqueness of generalized cospectral mates}
\subsection{ Preliminaries}
An orthogonal matrix $Q$ is called \emph{regular} if $Qe=e$ (or equivalently, $Q^\T e=e$). An old result of Johnson and Newman \cite{johnson1980JCTB} states that two graphs $G$ and $H$ are generalized cospectral if and only if there exists a regular orthogonal matrix $Q$ such that $Q^\T A(G)Q=A(H)$. A graph $G$ is \emph{controllable} if $W(G)$ is invertible. For controllable graphs, the corresponding matrix $Q$ is unique and rational.
\begin{lemma}\cite{johnson1980JCTB,wang2006EuJC}
	Let $G$ be a controllable graphs of order $n$ and $H$ be a graph generalized cospectral with $G$. Then $H$ is controllable and there exists a unique regular rational orthogonal matrix $Q$ such that $Q^\T A(G) Q=A(H)$. Moreover, the unique $Q$ satisfies $Q=W(G)W^{-1}(H)$ and hence is rational.
\end{lemma}
For a rational matrix $Q$, the \emph{level} of $Q$, denoted by $\ell(Q)$, or simply $\ell$, is the smallest positive integer $k$ such that $kQ$ is an integral matrix.
For a controllable graph $G$, define $\mathcal{Q}(G)$ to be the set of all regular rational orthogonal matrices  $Q$ such that $Q^\T A(G)Q$ is an adjacency matrix.  It is easy to show that for any $Q\in \mathcal{Q}(G)$, the level $\ell(Q)$ must be a factor of $d_n$, the last invariant factor of $W(G)$.  It turns out that under some mild  assumptions on $W$, some factors of $d_n$ can never be realized as  levels for any $Q\in \mathcal{Q}(G)$.  For nonzero integers $n$, $m$ and positive integer $k$, we use $m^k\mid\mid n$ to indicate that $m^k$ precisely divides $n$, i.e., $m^k\mid n$ but $m^{k+1}\nmid n$.
\begin{lemma}\cite{wang2013ElJC}\label{pl}
Let $Q\in \mathcal{Q}(G)$ with level $\ell$, and $p$ be an odd prime. Suppose that $\rank_p W=n-1$ and  $p\mid\mid \det W$ (or equivalently, $p\mid\mid d_n$). Then $p\nmid\ell$ and hence $\ell\mid\frac{d_n}{p}$.
\end{lemma}
\begin{lemma}\cite{wang2017JCTB}\label{tl}
Let $Q\in \mathcal{Q}(G)$ with level $\ell$. Suppose that $\rank_2 W=\lceil n/2\rceil$  and $2^{\lfloor n/2\rfloor}\mid\mid \det W$ (or equivalently, $2 \mid\mid d_n$). Then $2\nmid\ell$ and hence $\ell\mid\frac{d_n}{2}$.
\end{lemma}
Both of the above lemmas have been strengthened in a recent paper of Qiu et al. \cite{qiuarXiv}, using a new and unified approach. In particular, they  establish a stronger version of Lemma \ref{pl}  as follows.
\begin{lemma}\cite{qiuarXiv} Let $Q\in \mathcal{Q}(G)$ with level $\ell$, and $p$ be an odd prime. Suppose that $\rank_p W=n-1$ and  $p^k\mid\mid \det W$ (or equivalently  $p^k\mid\mid d_n$) for some positive integer $k$. Then $p^k\nmid\ell$ and hence $\ell\mid\frac{d_n}{p}$.
\end{lemma}
The following corollary is immediate.
\begin{corollary}\label{lp}
	Let $G\in \mathcal{F}_n^p$. Then $\ell(Q)=1$ or $\ell(Q)=p$ for any matrix $Q\in \mathcal{Q}(G)$.
\end{corollary}
Note that any regular rational orthogonal matrix with level one is a permutation matrix. Since permutation matrices  generate isomorphic graphs, we shall be mainly concerned with  the case $\ell(Q)=p$.
\begin{lemma}\label{rk1}
	Let $G\in \mathcal{F}_{n}^p$ and $Q\in \mathcal{Q}(G)$. If  $\ell(Q)=p$ then $\rank_p pQ=1$.
\end{lemma}
\begin{proof}
	Let $H$ be the graph such that $Q^\T A(G)Q=A(H)$. Then  $Q^\T W(G)=W(H)$, or equivalently,  $W^\T(G)Q=W^\T(H)$. Write $\hat{Q}=pQ$. We have $W^\T(G)\hat{Q}=pW^\T(H)\equiv 0\pmod{p}$. As $G\in \mathcal{F}_n^p$, we see that $\rank_p W(G)=n-1$ and hence  $\rank_p W^\T(G)=n-1$. Therefore, the solution space of  $W^\T(G)z\equiv 0\pmod{p}$ is one dimensional. Consequently, $\rank_p \hat{Q}\le 1$. On the other hand, since $\ell(Q)=p$, the minimality of $\ell(Q)$ means that $\hat{Q}$ contains at least one entry which is nonzero over $\mathbb{F}_p$. Thus  $\rank_p \hat{Q}\ge 1$. This proves that $\rank_p \hat{Q}= 1$, as desired.
\end{proof}
\subsection{Primitive matrix and its uniqueness}
We always assume that $p$ is an odd prime.
\begin{definition}\textup{
	A regular rational orthogonal matrix $Q$ of level $p$ is called a \emph{primitive matrix} if $\rank_p(pQ)=1$.  }
\end{definition}
\begin{example}\label{np4}\normalfont
	Consider a  regular rational orthogonal matrix
	\begin{equation}
	Q=\frac{1}{3}\left(\begin{matrix}-1&2&0&2\\0&0&3&0\\2&-1&0&2\\2&2&0&-1\end{matrix}\right).
	\end{equation}
One can see that  $\ell(Q)=3$ and $\rank_3(3Q)=1$. Thus, $Q$ is a primitive matrix.
\end{example}
Removing the second row and the third column from $Q$, the resulting submatrix is also a primitive matrix, each entry of which  is nonintegral. The following lemma summarizes this phenomenon in a slightly different manner.
\begin{lemma}\label{qdiag}
	Let $Q$ be a primitive matrix of order $n$. If there exists some entry which is nonintegral, then after necessary row permutations and column permutations, $Q$ has the  quasi-diagonal  form $\diag[Q_0,I]$, where  $Q_0$ is a primitive matrix containing none integral entries.
\end{lemma}
\begin{proof}
	Clearly for any primitive matrix $Q$, the integral entry of $Q$ can only be $0$ or $1$. Moreover, we claim that  $Q$ contains $0$ if and only if $Q$ contains $1$. The `if' part is clear since each row (and column) of $Q$ has  length one in $\mathbb{R}^n$. Let the $(i,j)$-entry $q_{i,j}$ of $Q$ is zero. Write $\hat{Q}=pQ$. Then either the $i$-th row or the $j$-th column of $\hat{Q}$ is the zero vector over $\mathbb{F}_p$ since otherwise one would easily find a $2\times 2$ invertible submatrix  in $\hat{Q}$, contradicting the fact that $\rank_p(\hat{Q})=1$. Clearly, in either case, $\hat{Q}$ contains $p$ as an entry, i.e., $Q$ contains $1$ as an entry.
	
	Suppose that $Q$ has exact $k$ entries equal one. Then these $k$ entries clearly lie in different rows and in different columns in $Q$, and all other entries in the involved rows and columns are necessarily zero. Thus, by some obvious row and column permutations, we can change $Q$ into a quasi-diagonal form
	$\diag[Q_0,I_k]$. Clearly, $Q_0$ is a primitive matrix. Finally, as $Q_0$ does not contain $1$ as an entry, it  does not contain $0$ as an entry. Thus, $Q_0$ contains none integral entries.
\end{proof}

\begin{definition}\normalfont
	Let $v$ be an $n$-dimensional integral vector and $Q$ be a primitive matrix. We say $Q$ can be \emph{generated} from $v$ (or $v$ can \emph{generate} $Q$) if each column of $pQ$ is a multiple of $v$ over $\mathbb{F}_p$.
\end{definition}
\begin{remark}\label{vcv}\normalfont
If a primitive matrix $Q$ can be generated from $v$ and $v'$ is an integral vector such that $v'\equiv cv\pmod{p}$ for some $c\not\equiv 0 \pmod{p}$, then $Q$ can also be generated from $v'$.
\end{remark}
Now we can give a necessary condition for a graph in $\mathcal{F}_n$ to have a generalized cospectral mate.
\begin{proposition}\label{gcmp}
	Let $G\in \mathcal{F}_n^p$. If  $G$ is not DGS then any nontrivial solution to $W^\T z\equiv 0\pmod{p}$ can generate some primitive matrix.
\end{proposition}
\begin{proof}
	As $G$ is not DGS, we see that $\mathcal{Q}(G)$ contains a matrix which is not a permutation matrix. Let $Q$ be such a matrix in  $\mathcal{Q}(G)$. Then by Corollary \ref{lp}, we have $\ell(Q)=p$. Moreover, by Lemma \ref{rk1}, we see that $\rank_p (pQ)=1$ and hence $Q$ is a  primitive matrix. From the proof of Lemma \ref{rk1}, we find that the column space of $pQ$ coincides with the (one-dimensional) solution space of $W^\T z\equiv 0\pmod{p}$. The proposition follows.
\end{proof}
\begin{remark}\label{zQ}\normalfont
	 Under the assumption of Proposition \ref{gcmp},  each matrix $Q\in \mathcal{Q}(G)$ that is not a permutation matrix can be generated from a nontrivial solution to $W^\T z\equiv 0\pmod{p}$.
\end{remark}
Suppose  $Q$ is a primitive matrix generated from $v$. By the very definition, we know that
all matrices obtained from $Q$ by column permutations can also be generated from $v$. A key result of this section is to show the reversed direction: Every primitive matrix generated from $v$ can be obtained from $Q$ by some column permutations.

The following lemma plays a fundamental role in this paper.
\begin{lemma}\label{ort}
	Let $u$ and $v$ be two $n$-dimensional integral column vectors with each entry nonzero modulo $p$. Suppose that (\rmnum{1}) $u$ and $v$ are linearly dependent over $\mathbb{F}_p$; (\rmnum{2}) $u\neq \pm v$; and (\rmnum{3}) $u^\T u=v^\T v=p^2$. Then $u^\T v=0$.
\end{lemma}
\begin{proof}
	Since $u$ and $v$ are linearly dependent over $\mathbb{F}_p$, there exist two integers $a$ and $b$, not both zero in $\mathbb{F}_p$, such that
	\begin{equation}\label{ld}
au+bv\equiv 0\pmod{p}.
	\end{equation} We claim that neither $a$ nor $b$ is zero. Actually, if $a\equiv 0\pmod{p}$ then $b\not\equiv 0\pmod{p}$ and hence we obtain $v\equiv 0\pmod{p}$ by \eqref{ld}. This contradicts our assumption on $v$. Thus  $a\not\equiv 0\pmod{p}$. Similarly, we also have  $b\not\equiv 0\pmod{p}$. This proves the claim.
	
	By \eqref{ld} we have $(au+bv)^\T (au+bv)\equiv 0\pmod{p^2}$, that is, $a^2u^\T u+2abu^\T v+b^2v^\T v \equiv 0\pmod{p^2}$, which can be reduced to
	\begin{equation}\label{abu}
2abu^\T v\equiv 0\pmod{p^2},
	\end{equation}
	as $u^\T u=v^\T v=p^2$. Since $2ab\not\equiv 0\pmod{p}$, Eq. \eqref{abu} can be further reduced to $u^\T v\equiv 0\pmod{p^2}$.
	
	By Cauchy-Schwartz inequality, we have  $|u^\T v|\le \sqrt{u^\T u} \sqrt{v^\T v}$ with the equality holding if and only if $u$ and $v$ are linear dependent over $\mathbb{R}$. From the last two conditions of this lemma, one clearly sees that  $u$ and $v$ are linear independent over $\mathbb{R}$. Therefore, we must have $|u^\T v|<p^2$, which, together with the established congruence $u^\T v\equiv 0\pmod{p^2}$, implies  $u^\T v=0$. This completes the proof.
\end{proof}
\begin{theorem}\label{uQ}
	Let  $Q$ be a primitive matrix generated from $v$. Then every primitive matrix generated from $v$ can be obtained from $Q$ by some column permutations.
\end{theorem}
\begin{proof}
	Let $Q'$ be any primitive matrix generated from $v$. We use $\alpha_i$ and $\beta_i$ respectively to denote the $i$-th column of $pQ$ and $pQ'$, for $i=1,2,\ldots,n$. We first consider the case that each entry of $v$ is nonzero modulo $p$.
	
	We claim that  $Q$ (and similarly $Q'$) contains none integral entries. Suppose to the contrary that $Q$ contains an integral entry. Then $Q$ contains one as an entry, say $q_{i,j}=1$.  Let $\alpha_k$ be the $k$-th column such that $\alpha_k\not\equiv 0\pmod{p}$. Then we have $\alpha_k=cv\pmod{p}$ for some integer $c$. Since $\alpha_k$ contains at least one entry which is nonzero modulo $p$, we see that $c\not\equiv 0\pmod{p}$. Consequently, as we assume that each entry of $v$ is nonzero modulo $p$, we find that each entry of $\alpha_k$ is nonzero modulo $p$, that is each entry of $\alpha_k/p$ is nonintegral. But, as $q_{i,j}=1$, the $i$-th row of $Q$ must be a standard unit vector and hence at least one entry of  $\alpha_k/p$ is integral. This contradiction proves the claim.
	
   We next claim that either $\alpha_i=\beta_j$ or $\alpha^\T_i\beta_j=0$ for each pair (possibly equal) $i$ and $j$. We may assume $\alpha_i\neq\beta_j$. Noting that $e^\T \alpha_i=e^\T \beta_j=p$, it can never happen that $\alpha_i\neq -\beta_j$. Thus, $\alpha_i\neq \pm \beta_j$. Moreover, as $Q$ and $Q'$ contains none integral entries, both $\alpha_i$ and $\beta_j$ are nonzero multiples of $z$ over $\mathbb{F}_p$. Thus, $\alpha_i$ and $\beta_j$ are linearly dependent over $\mathbb{F}_p$, and each entry of $\alpha_i$ and $\beta_j$ is nonzero modulo $p$. Of course,   $\alpha^\T_i\alpha_i=\beta^\T_j\beta_j=p^2$ by the orthogonality of $Q$ and $Q'$. Therefore, all conditions of Lemma \ref{ort} for $u=\alpha_i$ and $v=\beta_j$ are satisfied and we can obtain  $\alpha^\T_i\beta_j=0$, as claimed. Now we fix $\beta_j$ and consider all possible $\alpha_i$'s. Note that $\{\alpha_1,\alpha_2,\ldots,\alpha_n\}$ constitutes a basis of $\mathbb{R}^n$. Since $\beta_j$ is a nonzero vector in $\mathbb{R}^n$, the equality $\alpha_i^\T \beta_j=0$ cannot hold for all $i$ simultaneously. Thus, by the claim, $\beta_j=\alpha_i$ for some $i$.  That is, the $j$-th column of $pQ'$ must appear as a column of $pQ$. Noting that all column of   $pQ'$ are pairwise different, we see that $pQ'$ can be obtained from $pQ$, or equivalently,  $Q'$ can be obtained from $Q$,  by some column permutations.

   It remains to consider that case that $v$ contains  at least one entry which is zero modulo $p$. For convenience, we make a similar assumption on $v$ as in the proof of Lemma \ref{qdiag}. Assume all nonzero entries of $v$ appear as the first $k$ entries and  write
  $$v=\left(\begin{matrix}v_1\\v_2\end{matrix}\right),$$
  where $v_1$ is the $k$-dimensional column vector consisting of the first $k$ entries and $v_2$ is the $(n-k)$-dimensional zero vector. We remark that this assumption corresponds to the row permutations in Lemma \ref{qdiag}. Next we forbid row permutations and continue to use only column permutations to transform $Q$ and $Q'$ into a quasi-diagonal forms. It is not difficult to see that $Q$ and $Q'$ have like quasi-diagonal form, say $\diag[Q_0,I]$ and $\diag[Q'_0,I]$, where both $Q_0$ and $Q'_0$ have order $k$, the number of nonzero entries in $v$. Moreover, both $Q_0$ and $Q'_0$ can be generated from $v_1$. Using the conclusion for the first case, we see that $Q'_0$ can be obtained from $Q_0$ by some column permutations. Clearly, taking the same  column permutations on $\diag[Q_0,I]$ will result in  $\diag[Q'_0,I]$. Since $\diag[Q_0,I]$ and $\diag[Q'_0,I]$ are obtained from $Q$ and $Q'$ by some column permutations, we find that $Q'$ can be obtained from $Q$ by some column permutations.
\end{proof}
Before we present  the proof of Theorem \ref{atmost1}, we  would like to record the following fact from the above proof.
\begin{remark} \normalfont\label{intfree}
	If each entry of $v$  is nonzero $\pmod{p}$ and $Q$ is a primitive matrix generated from $v$, then each entry of $Q$ is nonintegral (or equivalently, nonzero).
	\end{remark}
\begin{proof}[Proof of Theorem \ref{atmost1}]
Let $G\in \mathcal{F}_n^p$. We may assume that $G$ is not DGS. Let $H_1$ and $H_2$ be any two generalized cospectral mates of $G$. It suffices to show that $H_1$ and $H_2$ are isomorphic. Let  $Q_1$ and $Q_2$ be the corresponding matrices such that $Q_1^\T A(G) Q=A(H_1)$ and $Q_2^\T A(G)Q_2=A(H_2)$.  Let $v$ be a nontrivial solution to $W^\T(G) z\equiv 0\pmod{p}$. Then by Corollary \ref{gcmp} and Remark \ref{zQ}, we see that both matrices $Q_1$ and $Q_2$  can be generated from $v$. It follows from Theorem \ref{uQ} that $Q_2=Q_1P$ for some permutation matrix $P$. Now, we have
$A(H_2)=Q_2^\T A(G)Q_2=Q_2^\T Q_1 A(H_1)Q_1^\T Q_2=P^\T A(H_1)P$, indicating that $H_1$ and $H_2$ are isomorphic. This completes the proof of Theorem \ref{atmost1}.
	\end{proof}
\subsection{0-1 property of $Q^\T AQ$}
The main aim of this subsection is to show that  the converse of Proposition \ref{gcmp} is also true.  We need an interesting and somewhat unexpected  result on adjacency matrix of simple graphs,  which may have independent interests. Roughly speaking, among all integral symmetric matrices, the subsets of all adjacency matrices are `closed' under generalized cospectrality. Here, the generalized spectrum of a matrix $A$ naturally refers to the spectrum of $A$ together with the spectrum of  $J-I-A$.
\begin{lemma}	\label{simcld}
	Let $A$ be an adjacency matrix and $B$ be an integral symmetric matrix. If $A$ and $B$ are generalized cospectral then $B$ must also be an adjacency matrix.
\end{lemma}
\begin{proof}
	For any $n\times n$ matrix $M$ and integer $k\in\{1,\ldots,n\}$, we use $c_k(M)$ to  denote the coefficient of the term $x^{n-k}$ in the characteristic polynomial $\det(xI-M)$ of $M$. Define $\xi(M)=c_2(M)+c_2(J-I-M)$. It is well known that $(-1)^kc_k(M)$ equals the sum of its principal minors of size $k$; see e.g. \cite[Theorem 1.2.16]{hj2012}. When $k=2$ and $M$ is an adjacency matrix, $c_2(M)$ equals the opposite of the number of edges in the corresponding graph. Since the total number of edges in a graph and in its complement is the constant $\binom{n}{2}$, we have
	\begin{equation}\label{xia}
	\xi(A)=-\binom{n}{2}.
	\end{equation}
	Write $B=(b_{i,j})_{n\times n}$. Note that $\tr(A)=0$ as each diagonal entry of $A$ is zero.  Since $A$ and $B$ are cospectral, we have $\tr(B)=0$, that is,
	\begin{equation}\label{tr0}
	\sum_{1\le k\le n}b_{k,k}=0.
	\end{equation}
	Next we estimate $\xi(B)$,  the sum of $c_2(B)$ and $c_2(J-I-B)$. Noting that  $B$ is symmetric, we have
	\begin{equation}
	c_2(B)=\sum_{1\le i<j\le n}\left|\begin{matrix}
	b_{i,i}&b_{i,j}\\
	b_{j,i}&b_{j,j}\\
	\end{matrix}\right|=\sum_{1\le i<j\le n} b_{i,i}b_{j,j}-\sum_{1\le i<j\le n}b^2_{i,j},
	\end{equation}
	and similarly,
	\begin{equation}
	c_2(J-I-B)=\sum_{1\le i<j\le n}\left|\begin{matrix}
	-b_{i,i}&1-b_{i,j}\\
	1- b_{j,i}&-b_{j,j}\\
	\end{matrix}\right|=\sum_{1\le i<j\le n} b_{i,i}b_{j,j}-\sum_{1\le i<j\le n}(1-b_{i,j})^2.
	\end{equation}
	Adding the above two equalities and using \eqref{tr0}, we find
	\begin{eqnarray}
	\xi(B)&=&2\sum_{1\le i<j\le n}b_{i,i}b_{j,j}-\sum_{1\le i<j\le n}\left(b^2_{i,j}+(1-b_{i,j})^2\right)\nonumber\\
	&=&\left(\sum_{1\le k\le n}b_{k,k}\right)^2-\sum_{1\le k\le n}b^2_{k,k}-\sum_{1\le i<j\le n}\left(2\left(b_{i,j}-\frac{1}{2}\right)^2+\frac{1}{2}\right)\nonumber\\
	&\le&-\sum_{1\le i<j\le n}\left(2\left(b_{i,j}-\frac{1}{2}\right)^2+\frac{1}{2}\right)\label{inq1},
	\end{eqnarray}
	where equality holds in \eqref{inq1} if and only if $\sum_{1\le k\le n}b^2_{k,k}=0$, i.e., all diagonals of $B$ are zero.
	
	Consider the quadratic function $f(x)=2(x-1/2)^2+1/2$, $x\in \mathbb{Z}$. It is easy to see that  $f(x)\ge 1$ for all $x\in \mathbb{Z}$, and the  equality holds if and only if $x\in\{0,1\}$. Since $B$ is integral, we have
	\begin{equation}\label{inq2}
	-\sum_{1\le i<j\le n}\left(2\left(b_{i,j}-\frac{1}{2}\right)^2+\frac{1}{2}\right)\le -\sum_{1\le i<j\le n} 1= -\binom{n}{2},
	\end{equation}
	with equality holding if and only if each non-diagonal entry $b_{i,j}$ is 0 or 1.  Finally, as $A$ and $B$ are generalized cospectral, we must have $\xi(A)=\xi(B)$ and hence $\xi(B)=-\binom{n}{2}$ by \eqref{xia}.
	This means that the equalities must hold in \eqref{inq1} and \eqref{inq2} simultaneously. Using the established conditions for these two equalities, we find that the symmetric matrix $B$ is a $(0,1)$-matrix with vanishing diagonal. This completes the proof of this lemma.
\end{proof}
\begin{lemma}\cite{wang2019LAA}\label{mustint}
	Let $G$ be a controllable graph with $n$ vertices. Let $p$ be an odd prime. Suppose that $p^2\mid \det W$ and $\rank_p W=n-1$. Let $v$ be a nontrivial integral solution to $W^\T z\equiv 0\pmod{p}$. If there exists a primitive matrix $Q$ generated from $v$, then $Q^\T A Q$ is an integral matrix.
\end{lemma}
Now we can show that the necessary condition for $G$ to have a generalized cospectral mate is also sufficient.
\begin{theorem}\label{gcm_pvec}
	Let $G\in \mathcal{F}_n^p$. Then  $G$ is not DGS if and only if any nontrivial solution to $W^\T z\equiv 0\pmod{p}$ can generate some primitive matrix.
\end{theorem}
\begin{proof}
	It suffices to show the sufficiency part. Let $v$ be a nontrivial solution to $W^\T v\equiv 0\pmod{p}$. Let $Q$ be a primitive matrix generated from $v$. Clearly, $G$ satisfies the condition of Lemma \ref{mustint}. Thus, $Q^\T A Q$ is an integral matrix. Note that  $Q^\T A Q$ is generalized cospectral with $A$.  It follows from  Lemma \ref{simcld} that $Q^\T A Q$ is the adjacency matrix of some graph, say $H$. Thus, $G$ is not DGS. This proves the theorem.
\end{proof}
	
\section{Finding generalized cospectral mates}
We shall develop an algorithm to determine whether a given graph $G\in \mathcal{F}_n^p$ is DGS. And when the graph $G$ is not DGS, the algorithm will find its (unique) generalized cospectral mate. The overall idea is based on Theorem \ref{gcm_pvec}. We pick an arbitrary nontrivial solution $v$ of $W^\T(G)z\equiv 0\pmod{p}$ and try to generate a primitive matrix.  By Lemma \ref{qdiag}, it suffices to consider the restricted case that every entry of $v$ is nonzero modulo $p$. Indeed, for general $v$, we use $v^*$ to denote the vector obtained from $v$ by deleting the zero entries. Then it is easy to see that $v$ can generate a primitive matrix if and only $v^*$ can do so. To generate a primitive matrix $Q$ from $v$ with $t$ zero entries, we first use $v^*$ to generate a  primitive matrix $Q_0$ of order $n-t$. Then we can obtain an $n\times n$ primitive matrix $Q$ from $Q_0$ by adding $t$ 1's and appropriate number of 0's naturally
as in Example \ref{np4}.

\subsection{Criterion for a vector to generate a primitive matrix}
The following lemma gives a simple condition for an integral vector to generate some primitive matrix. It essentially appeared in \cite{wang2006EuJC}. We include its short proof here.
\begin{lemma}\cite{wang2006EuJC}
	Let $v$ be an $n$-dimensional integral vector. If $v$ can generated some primitive matrix $Q$, then $v^\T e\equiv 0\pmod{p}$ and $v^\T v\equiv 0\pmod{p}$.
\end{lemma}
\begin{proof}
	Let $\hat{Q}=pQ$ and $u$ be a column of $\hat{Q}$ such that $u\not\equiv 0\pmod{p}$. By the condition of this lemma, there exists an integer $c$ such that $u\equiv cv\pmod{p}$. As $u\not\equiv 0\pmod{p}$, we must have $c\not\equiv 0\pmod{p}$. Let $d$ be an integer such that $cd \equiv 1\pmod{p}$. Then we have $v\equiv du\pmod{p}$. Noting that  $\hat{Q}^\T e=pQ^\T e=pe$ and $\hat{Q}^\T \hat{Q}=p^2I$, we have $u^\T e=p$ and  $u^\T u=p^2$. Thus, $v^\T e\equiv du^\T e\equiv 0\pmod{p}$ and  $v^\T v\equiv d^2 u^\T u\equiv 0\pmod{p}$. This proves the lemma.
	\end{proof}
\begin{definition}\label{ppr}\textup{
	For two integral vectors $v$ and $w$, we call $w$  a \emph{perfect $p$-representative} of $v$ if $w$ satisfies (\rmnum{1}) $w\equiv v\pmod{p}$, (\rmnum{2})  $w^\T e=p$, and (\rmnum{3}) $w^\T w=p^2$.}
\end{definition}
\begin{proposition}\label{do}
	Let $v$ be an integral vector with each entry nonzero modulo $p$. Let $c_1$ and $c_2$ be two distinct integers in $\{1,2,\ldots,p-1\}$. Then the followings hold:\\	
	(\rmnum{1}) Any two  distinct perfect $p$-representatives  $u_1$ and $u_2$ of $c_1 v$ are orthogonal in $\mathbb{R}^n$. \\	
	(\rmnum{2}) Any two perfect $p$-representatives  $w_1$ and $w_2$ of $c_1v$ and $c_2v$ respectively are distinct and orthogonal in $\mathbb{R}^n$.
\end{proposition}
\begin{proof}
Note that $u_1\equiv u_2\equiv c_1v\pmod{p}$. The assumptions on $v$ and $c_1$ imply that each entry of $u_1$ and $u_2$ is nonzero modulo $p$. 	As $u_1\equiv u_2\pmod{p}$, we see that $u_1$ and $u_2$ are clearly linearly dependent over $\mathbb{F}_p$. Since $u_1^\T e=u_2^\T e=p$, we must have $u_1\neq -u_2$ and hence $u_1\neq \pm u_2$ as $u_1$ and $u_2$ are distinct. Noting that $u_1^\T u_1=u_2^\T u_2=p^2$ and using Lemma \ref{ort} we have $u_1^\T u_2=0$. This proves (\rmnum{1}).

By the assumptions on $v$,  $c_1$ and $c_2$, we see that $c_1v\not\equiv c_2v\pmod{p}$. Noting that $w_1\equiv c_1v$ and $w_2\equiv c_2v$, we have $w_1\not\equiv w_2$ and hence $w_1\neq w_2$. Now, (\rmnum{2}) holds by a  similar argument  as in (\rmnum{1}).
\end{proof}
\begin{theorem}
	Let $v$ be an $m$-dimensional integral vector with each entry nonzero modulo $p$. For each $k\in \{1,2,\ldots,p-1\}$, let $\mathcal{R}_k$ denote the collection of all perfect $p$-representatives of $kv$. Then $v$ can generate a primitive matrix if and only if $\sum_{k=1}^{p-1}|\mathcal{R}_k|=m$.
\end{theorem}
\begin{proof}
	Suppose $Q$ is a primitive matrix generated from $v$. Since every entry of $v$ is nonzero (mod $p$), we see from Remark \ref{intfree} that $Q$ contains no integral entries. Let $z$ be any column of $pQ$. Then each entry of $z$ is nonzero modulo $p$ and of course $z\not\equiv 0\pmod{p}$. Thus $z\equiv kv\pmod{p}$ for $k\in \{1,2,\ldots,p-1\}$. Consequently, by the regularity and orthogonality of $Q$, we find that $z$ is a perfect $p$-representative of $kv$, i.e., $z\in \mathcal{R}_k$. Thus, we have $\sum_{k=1}^{p-1}|\mathcal{R}_k|\ge m$. By Proposition \ref{do},  all these $p-1$ sets $\mathcal{R}_k$'s are disjoint and any two vectors in $\cup_{k=1}^{p-1}\mathcal{R}_k$ are orthogonal in $\mathbb{R}^m$. Thus the strict inequality $\sum_{k=1}^{p-1}|\mathcal{R}_k|> m$ can never hold and hence $\sum_{k=1}^{p-1}|\mathcal{R}_k|=m$.
	
Suppose $\sum_{k=1}^{p-1}|\mathcal{R}_k|=m$. We construct an integral matrix $\hat{Q}$ using all the $m$ vectors in $\cup_{k=1}^{p-1}\mathcal{R}_k$. Using Definition \ref{ppr} and Proposition  \ref{do}, we can check that  $\frac{1}{p}\hat{Q}$ is a primitive matrix and $\frac{1}{p}\hat{Q}$ is generated by $v$.
\end{proof}
\subsection{Constructing all perfect $p$-representatives}
\begin{definition}\normalfont
		For two integral vectors $v$ and $w$, we call $w$  the \emph{shortest $p$-representative} of $v$ if  $w\equiv v\pmod{p}$ and $|w_i|\le \frac{p-1}{2}$ for each entry $w_i$ of $w$.
\end{definition}
\begin{remark}\normalfont
For a given integral vector $v$, there may be none, unique or many perfect $p$-representatives of $v$. Nevertheless, the shortest $p$-representative of $v$ always exists and is unique. Also note that the shortest $p$-representative of $v$ has the shortest Euclidian length among all vectors that are congruent to $v$ modulo $p$.
\end{remark}
\begin{example}\normalfont
Let  $n=6$, $p=3$,
$$v=\left(\begin{matrix}
2\\2\\2\\1\\1\\1\end{matrix}\right)
\quad\text{and}\quad
\hat{Q}=\left(\begin{matrix}
2&-1&-1&1&1&1\\
-1&2&-1&1&1&1\\
-1&-1&2&1&1&1\\
1&1&1&2&-1&-1\\
1&1&1&-1&2&-1\\
1&1&1&-1&-1&2
\end{matrix}\right).$$
Then the shortest $3$-representative of $v$ is $(-1,-1,-1,1,1,1)^\T$. All the first 3 columns of $\hat{Q}$ are perfect  $3$-representatives of $v$, while the remaining three columns are perfect  $3$-representatives of $2v$.
\end{example}
The next lemma indicates that all perfect $p$-representatives of a vector $v$ are very close to its shortest $p$-representative in the sense of Hamming distance. Recall that the Hamming distance of two vectors are the number of positions in which they differ.
\begin{lemma}\label{atm3}
	 For an integer vector $v$ with each entry nonzero modulo $p$, let $w$ be a perfect $p$-representative and $u$ be the shortest $p$-representative of $v$. Then the Hamming distance of $w$ and $u$ is at most 3. Moreover, for any index $i$ such that $w_i\neq u_i$, either (\rmnum{1}) $w_i=u_i+p$ and $u_i< 0$, or (\rmnum{2}) $w_i=u_i-p$ and $u_i>0$.
\end{lemma}
\begin{proof}
	Let $i$ be an index such that $w_i\neq u_i$. As $w_i\equiv u_i\pmod{p}$ and $|u_i|\le \frac{p-1}{2}$, it is not difficult to see that $|w_i|\ge \frac{p+1}{2}$ and consequently, $w^2_i>\frac{p^2}{4}$. Therefore, there are at most 3 different such indices as $w^\T w=p^2$.  This proves the first part of this lemma. Note that $w_i=u_i+kp$ for some integer $k$. It is easy to verify the  remaining part using the obvious restriction that $|w_i|<p$.
\end{proof}
  The following proposition is immediate from Lemma \ref{atm3}.
\begin{proposition}
	 For an integer vector $v$ with each entry nonzero modulo $p$,  let $u$ be the shortest $p$-representative of $v$. If $v$ has at least one perfect $p$-representative, then  $|u^\T e-p|\le 3p$ and $u^\T u\le p^2$.
\end{proposition}
Using the entry sum of the shortest $p$-representative of a vector $v$, we can know more about its perfect $p$-representatives. Let $e_i$ denote the $i$-th standard unit vector in $\mathbb{R}^n$.
\begin{lemma}\label{adm}
	 For an integer vector $v$  with each entry nonzero modulo $p$, let $u$ be the shortest $p$-representative of $v$. Suppose that $u^\T e-p=sp$ for some $s\in \{-3,-2,\ldots,3\}$ and $u^\T u\le p^2$. Then any perfect $p$-representative $w$ of $v$ can be written as
	 \begin{equation}\label{wu}
	w=u+\sum_{i\in I}pe_i-\sum_{j\in D}pe_j,
	 \end{equation}
	 where $I$ and $D$ are disjoint (possibly empty) subsets of $\{1,2,\ldots,n\}$ satisfying the following conditions:\\
	 (\rmnum{1}) $|I|+|D|\le 3$;\\
	 (\rmnum{2}) $|D|-|I|=s$;\\
	 (\rmnum{3})  $u_i<0$ for each $i\in I$ and $u_j>0$ for $j\in D$; and\\
	 (\rmnum{4}) $\sum_{k\in I\cup D}|u_k|=\frac{1}{2p}u^\T u+\frac{p}{2}(|I|+|D|-1).$
\end{lemma}
\begin{proof}
	By Lemma \ref{atm3}, we know that $w$ can be expressed as in \eqref{wu} where $I$ and $D$ are disjoint subsets of $\{1,2,\ldots,n\}$ satisfying (\rmnum{1}) and  (\rmnum{3}). By \eqref{wu}, we have
	$$w^\T e=u^\T e +p(|I|- |D|)=p+sp+p(|I|- |D|).$$
	Thus, $w^\T e=p$ is equivalent to (\rmnum{2}). It remains to check  (\rmnum{4}).  Due to  (\rmnum{3}),  we can rewrite  \eqref{wu} as
	\begin{equation}
	w=u-p\sum_{k\in I\cup D}\sgn(u_k)e_k.
	\end{equation}
	Since  $e_k$'s are standard unit vectors, we have
	\begin{equation}
w^\T w=u^\T u+p^2\sum_{k\in I\cup D} e_k^\T e_k-2p\sum_{k\in I\cup D} \sgn(u_k)u^\T e_k=u^\T u+p^2(|I|+|D|)-2p\sum_{k\in I\cup D} |u_k|.
	\end{equation}
Now it is straightforward to see that $w^\T w=p^2$ if and only if (\rmnum{4}) holds.
\end{proof}
The following table gives a more visual description of Lemma \ref{adm}. We may call the desired set $I$ (resp. $D$) an \emph{increasing} subset (resp. \emph{decreasing} subset). For any $s\in \{-3,-2,\ldots,3\}$, all possible pairs $(|I|,|D|)$ for the sizes of $I$ and $D$ are rather restricted due to (\rmnum{1}) and (\rmnum{2}). For example, when $s=-3$, we must have $(|I|,|D|)=(3,0)$; when $s=-1$, we have $(|I|,|D|)=(1,0)$, or $(2,1)$. In Table 1, we use `$\uparrow\uparrow\uparrow$' to denote an adjustment strategy corresponding to the case $(|I|,|D|)=(3,0)$, that is, $w$ is obtained from $u$ by a fixed increasement on three (negative) entries. The expression $\frac{1}{2p}u^\T u+p$  attached to `$\uparrow\uparrow\uparrow$'  corresponds to (\rmnum{4}) as $|I|=3$ and $|D|$=0. The symbol `$--$' at the middle of Table \ref{allperfect} means $I=D=\emptyset$. This only happens when the length of the shortest $p$-representative is exactly $p$.
\begin{table}[htbp]
	\renewcommand\arraystretch{1.5}
	\centering
	\caption{\label{allperfect} Generating all possible perfect $p$-representatives from the shortest one}
	\begin{tabular}{c|c}
		\toprule
		$\frac{1}{p}(u^\T e-p)$ &adjustment strategy\\
		\midrule
		$-3$&$\uparrow\uparrow\uparrow$($\frac{1}{2p}u^\T u+p$)\\
		$-2$&$\uparrow\uparrow$($\frac{1}{2p}u^\T u+\frac{p}{2}$)\\
		$-1$&$\uparrow$($\frac{1}{2p}u^\T u$)\quad or \quad $\uparrow\uparrow\downarrow$($\frac{1}{2p}u^\T u+p$) \\
						
		$0$&$--$($\frac{1}{2p}u^\T u-\frac{p}{2}=0$) \quad or \quad $\uparrow\downarrow$($\frac{1}{2p}u^\T u+\frac{p}{2}$)\\
		$1$&$\downarrow$($\frac{1}{2p}u^\T u$)\quad or \quad $\uparrow\downarrow\downarrow$($\frac{1}{2p}u^\T u+p$)\\
		$2$&$\downarrow\downarrow$($\frac{1}{2p}u^\T u+\frac{p}{2}$)\\
		$3$&$\downarrow\downarrow\downarrow$($\frac{1}{2p}u^\T u+p$)\\
		\bottomrule
	\end{tabular}
\end{table}
\subsection{The algorithm}
Now we can summarize the above discussions in the following algorithm.
\begin{algorithm}[t]
\caption{Finding generalized cospectral mate}
\hspace*{0.02in} {\bf Input:} 
a graph $G\in\mathcal{F}_n^p$.\\
\hspace*{0.02in} {\bf Output:} 
DGS or the unique generalized cospectral mate of $G$.
\begin{algorithmic}[1]
\State Compute any nontrivial solution $v$ to $W^\T z\equiv 0\pmod{p}$.   \;

\If {$v^\T v\equiv 0\pmod{p}$}
\State Set $S:=\{i\colon\,v_i\equiv0\pmod{p}\}$.\;
\State Remove zero entries $v_i, i\in S$ from  $v$ to obtain  a vector $v^*$.\;
\State Set $\mathcal{R}:=\emptyset$.
\For{ $k$ from $1$ to $p-1$} 
\State  Compute the shortest $p$-representative $u$ of $kv^*$.\;
\If{$|\frac{1}{p}(u^\T e-p)|\le 3$ and $u^\T u\le p^2$}
\State Find all possible perfect $p$-representatives  from $u$ by Table \ref{allperfect}.
\State Update $\mathcal{R}$ by  appending all  perfect $p$-representatives of $kv^*$.
\If{$|\mathcal{R}|=n-|S|$}
\State Construct a primitive matrix $Q$ using $\mathcal{R}$ and unit vectors $e_i$'s, $i\in S$.
\State \Return graph $H$ with adjacency matrix  $Q^\T A(G)Q$.
\EndIf
\EndIf
\EndFor
\EndIf
\State \Return DGS.
\end{algorithmic}
\end{algorithm}
We give two examples to illustrate Algorithm 1.
\begin{example}\label{exmgcd}\normalfont
Let $n$=16 and $G$ be the graph  with adjacency matrix
$$ A=\scriptsize{
\left(
\begin{array}{cccccccccccccccc}
	0 & 1 & 0 & 1 & 1 & 1 & 1 & 0 & 0 & 1 & 1 & 1 & 1 & 1 & 0 & 0 \\
	1 & 0 & 0 & 1 & 1 & 0 & 0 & 0 & 0 & 0 & 0 & 1 & 0 & 0 & 1 & 1 \\
	0 & 0 & 0 & 0 & 1 & 1 & 0 & 0 & 0 & 1 & 1 & 1 & 1 & 1 & 0 & 0 \\
	1 & 1 & 0 & 0 & 1 & 0 & 0 & 1 & 1 & 0 & 0 & 1 & 1 & 0 & 0 & 1 \\
	1 & 1 & 1 & 1 & 0 & 0 & 0 & 0 & 0 & 0 & 1 & 0 & 0 & 0 & 1 & 1 \\
	1 & 0 & 1 & 0 & 0 & 0 & 1 & 1 & 1 & 1 & 0 & 0 & 0 & 0 & 0 & 0 \\
	1 & 0 & 0 & 0 & 0 & 1 & 0 & 0 & 0 & 0 & 1 & 0 & 0 & 1 & 0 & 0 \\
	0 & 0 & 0 & 1 & 0 & 1 & 0 & 0 & 1 & 1 & 0 & 0 & 0 & 0 & 0 & 0 \\
	0 & 0 & 0 & 1 & 0 & 1 & 0 & 1 & 0 & 0 & 0 & 0 & 0 & 0 & 0 & 0 \\
	1 & 0 & 1 & 0 & 0 & 1 & 0 & 1 & 0 & 0 & 1 & 1 & 1 & 0 & 1 & 1 \\
	1 & 0 & 1 & 0 & 1 & 0 & 1 & 0 & 0 & 1 & 0 & 1 & 0 & 1 & 0 & 1 \\
	1 & 1 & 1 & 1 & 0 & 0 & 0 & 0 & 0 & 1 & 1 & 0 & 1 & 1 & 0 & 1 \\
	1 & 0 & 1 & 1 & 0 & 0 & 0 & 0 & 0 & 1 & 0 & 1 & 0 & 0 & 1 & 1 \\
	1 & 0 & 1 & 0 & 0 & 0 & 1 & 0 & 0 & 0 & 1 & 1 & 0 & 0 & 1 & 0 \\
	0 & 1 & 0 & 0 & 1 & 0 & 0 & 0 & 0 & 1 & 0 & 0 & 1 & 1 & 0 & 0 \\
	0 & 1 & 0 & 1 & 1 & 0 & 0 & 0 & 0 & 1 & 1 & 1 & 1 & 0 & 0 & 0 \\
\end{array}
\right)}.
$$ Using Mathematica, we can find that $G\in \mathcal{F}_n^p$ for $p=5$. Indeed,  the last invariant factor of $W$ is $$d_n=2\times 5^2\times 11\times 41\times 28573\times 260723\times71447889577.$$
A nontrivial solution to $W^\T z\equiv 0\pmod{5}$ is $v=(4, 0, 0, 0, 0, 0, 2, 1, 2, 1, 0, 0, 2, 2, 0, 1)^\T$. Clearly, $v^\T v\equiv 0\pmod{5}$. Now $S=\{2,3,4,5,6,11,12,15\}$,  the indices for the zero entries of $v$. Remove these zero entries we obtain $v^*=(4, 2, 1, 2, 1, 2, 2, 1)^\T$. Table \ref{allp5} illustrates the iterations of the \emph{for} loop.
\begin{table}[htbp]\label{tbexa}
	\small
	\renewcommand\arraystretch{1.5}
	\centering
	\caption{\label{allp5} Sufficient  perfect $5$-representatives}
	\begin{tabular}{c|c|c|c|c}
	\toprule
	$k$ &shortest $5$-representative $u$&$\frac{1}{5}(u^\T e-5)$&$u^\T u$&perfect $p$-representatives\\
	\midrule
\multirow{4}{*}{$1$}&
\multirow{4}{*}{$(-1,2,1,2,1,2,2,1)^\T$}&
\multirow{4}{*}{$1$}&\multirow{4}{*}{$20$} &$(-1, -3, 1, 2, 1, 2, 2, 1)^\T$
	\\
	& &  &&$(-1, 2, 1, -3, 1, 2, 2, 1)^\T$ \\
	& &  &&$(-1, 2, 1, 2, 1, -3, 2, 1)^\T$ \\
	& &  &&$(-1, 2, 1, 2, 1, 2, -3, 1)^\T$ \\
	\hline

\multirow{1}{*}{$2$}&
\multirow{1}{*}{$(-2, -1, 2, -1, 2, -1, -1, 2)^\T$}&
\multirow{1}{*}{$-1$}&\multirow{1}{*}{$20$}& $(3, -1, 2, -1, 2, -1, -1, 2)^\T$\\
\hline
\multirow{3}{*}{$3$}&
\multirow{3}{*}{$(2, 1, -2, 1, -2, 1, 1, -2)^\T$}&
\multirow{3}{*}{$-1$}&\multirow{3}{*}{$20$}& $(2, 1, 3, 1, -2, 1, 1, -2)^\T$ \\
& &  &&$(2, 1, -2, 1, 3, 1, 1, -2)^\T$ \\
& &  &&$(2, 1, -2, 1, -2, 1, 1, 3)^\T$ \\
	\bottomrule
\end{tabular}
\end{table}
In the third  iteration, $|\mathcal{R}|$ reaches 8, which is the dimension of $v^*$. Now, using  $\mathcal{R}$ and $S$, we can  construct $$
Q=\scriptsize{\frac{1}{5}\left(
\begin{array}{cccccccccccccccc}
-1 & -1 & -1 & -1 & 3 & 2 & 2 & 2 & 0 & 0 & 0 & 0 & 0 & 0 & 0 & 0 \\
0 & 0 & 0 & 0 & 0 & 0 & 0 & 0 & 5 & 0 & 0 & 0 & 0 & 0 & 0 & 0 \\
0 & 0 & 0 & 0 & 0 & 0 & 0 & 0 & 0 & 5 & 0 & 0 & 0 & 0 & 0 & 0 \\
0 & 0 & 0 & 0 & 0 & 0 & 0 & 0 & 0 & 0 & 5 & 0 & 0 & 0 & 0 & 0 \\
0 & 0 & 0 & 0 & 0 & 0 & 0 & 0 & 0 & 0 & 0 & 5 & 0 & 0 & 0 & 0 \\
0 & 0 & 0 & 0 & 0 & 0 & 0 & 0 & 0 & 0 & 0 & 0 & 5 & 0 & 0 & 0 \\
-3 & 2 & 2 & 2 & -1 & 1 & 1 & 1 & 0 & 0 & 0 & 0 & 0 & 0 & 0 & 0 \\
1 & 1 & 1 & 1 & 2 & 3 & -2 & -2 & 0 & 0 & 0 & 0 & 0 & 0 & 0 & 0 \\
2 & -3 & 2 & 2 & -1 & 1 & 1 & 1 & 0 & 0 & 0 & 0 & 0 & 0 & 0 & 0 \\
1 & 1 & 1 & 1 & 2 & -2 & 3 & -2 & 0 & 0 & 0 & 0 & 0 & 0 & 0 & 0 \\
0 & 0 & 0 & 0 & 0 & 0 & 0 & 0 & 0 & 0 & 0 & 0 & 0 & 5 & 0 & 0 \\
0 & 0 & 0 & 0 & 0 & 0 & 0 & 0 & 0 & 0 & 0 & 0 & 0 & 0 & 5 & 0 \\
2 & 2 & -3 & 2 & -1 & 1 & 1 & 1 & 0 & 0 & 0 & 0 & 0 & 0 & 0 & 0 \\
2 & 2 & 2 & -3 & -1 & 1 & 1 & 1 & 0 & 0 & 0 & 0 & 0 & 0 & 0 & 0 \\
0 & 0 & 0 & 0 & 0 & 0 & 0 & 0 & 0 & 0 & 0 & 0 & 0 & 0 & 0 & 5 \\
1 & 1 & 1 & 1 & 2 & -2 & -2 & 3 & 0 & 0 & 0 & 0 & 0 & 0 & 0 & 0 \\
\end{array}
\right)}.
$$
Now $Q^\T AQ$ gives the adjacency matrix for the generalized cospectral mate of $G$. Indeed, direct computation shows that $$Q^\T A Q=\scriptsize{\left(
		\begin{array}{cccccccccccccccc}
		0 & 0 & 0 & 1 & 1 & 0 & 0 & 0 & 0 & 1 & 1 & 0 & 0 & 0 & 1 & 1 \\
		0 & 0 & 0 & 0 & 1 & 0 & 1 & 1 & 0 & 1 & 0 & 0 & 0 & 1 & 1 & 1 \\
		0 & 0 & 0 & 0 & 0 & 1 & 0 & 0 & 0 & 0 & 0 & 0 & 1 & 1 & 0 & 0 \\
		1 & 0 & 0 & 0 & 1 & 0 & 0 & 0 & 0 & 0 & 1 & 0 & 1 & 0 & 0 & 0 \\
		1 & 1 & 0 & 1 & 0 & 0 & 1 & 0 & 1 & 0 & 1 & 1 & 1 & 1 & 1 & 0 \\
		0 & 0 & 1 & 0 & 0 & 0 & 1 & 0 & 0 & 0 & 1 & 0 & 1 & 0 & 0 & 0 \\
		0 & 1 & 0 & 0 & 1 & 1 & 0 & 1 & 0 & 1 & 0 & 0 & 1 & 1 & 1 & 1 \\
		0 & 1 & 0 & 0 & 0 & 0 & 1 & 0 & 1 & 0 & 1 & 1 & 0 & 1 & 1 & 0 \\
		0 & 0 & 0 & 0 & 1 & 0 & 0 & 1 & 0 & 0 & 1 & 1 & 0 & 0 & 1 & 1 \\
		1 & 1 & 0 & 0 & 0 & 0 & 1 & 0 & 0 & 0 & 0 & 1 & 1 & 1 & 1 & 0 \\
		1 & 0 & 0 & 1 & 1 & 1 & 0 & 1 & 1 & 0 & 0 & 1 & 0 & 0 & 1 & 0 \\
		0 & 0 & 0 & 0 & 1 & 0 & 0 & 1 & 1 & 1 & 1 & 0 & 0 & 1 & 0 & 1 \\
		0 & 0 & 1 & 1 & 1 & 1 & 1 & 0 & 0 & 1 & 0 & 0 & 0 & 0 & 0 & 0 \\
		0 & 1 & 1 & 0 & 1 & 0 & 1 & 1 & 0 & 1 & 0 & 1 & 0 & 0 & 1 & 0 \\
		1 & 1 & 0 & 0 & 1 & 0 & 1 & 1 & 1 & 1 & 1 & 0 & 0 & 1 & 0 & 0 \\
		1 & 1 & 0 & 0 & 0 & 0 & 1 & 0 & 1 & 0 & 0 & 1 & 0 & 0 & 0 & 0 \\
		\end{array}
		\right).}$$
\end{example}
\begin{example}\label{exmdgs}\normalfont
	Let $n$=16 and $G$ be the graph  with adjacency matrix
	$$A=\scriptsize{\left(
	\begin{array}{cccccccccccccccc}
	0 & 1 & 0 & 0 & 1 & 1 & 0 & 1 & 0 & 1 & 0 & 1 & 1 & 1 & 1 & 1 \\
	1 & 0 & 1 & 0 & 1 & 0 & 0 & 1 & 0 & 0 & 1 & 1 & 0 & 0 & 1 & 0 \\
	0 & 1 & 0 & 1 & 1 & 0 & 1 & 1 & 1 & 1 & 1 & 0 & 0 & 1 & 0 & 1 \\
	0 & 0 & 1 & 0 & 1 & 1 & 1 & 0 & 1 & 0 & 1 & 0 & 1 & 1 & 0 & 0 \\
	1 & 1 & 1 & 1 & 0 & 1 & 1 & 1 & 1 & 0 & 1 & 1 & 1 & 0 & 0 & 1 \\
	1 & 0 & 0 & 1 & 1 & 0 & 1 & 0 & 0 & 0 & 1 & 0 & 0 & 1 & 0 & 1 \\
	0 & 0 & 1 & 1 & 1 & 1 & 0 & 0 & 1 & 0 & 0 & 1 & 0 & 0 & 0 & 1 \\
	1 & 1 & 1 & 0 & 1 & 0 & 0 & 0 & 1 & 1 & 1 & 1 & 0 & 0 & 0 & 0 \\
	0 & 0 & 1 & 1 & 1 & 0 & 1 & 1 & 0 & 1 & 0 & 0 & 0 & 1 & 0 & 0 \\
	1 & 0 & 1 & 0 & 0 & 0 & 0 & 1 & 1 & 0 & 0 & 0 & 1 & 1 & 1 & 1 \\
	0 & 1 & 1 & 1 & 1 & 1 & 0 & 1 & 0 & 0 & 0 & 0 & 0 & 0 & 0 & 1 \\
	1 & 1 & 0 & 0 & 1 & 0 & 1 & 1 & 0 & 0 & 0 & 0 & 1 & 1 & 1 & 1 \\
	1 & 0 & 0 & 1 & 1 & 0 & 0 & 0 & 0 & 1 & 0 & 1 & 0 & 1 & 0 & 0 \\
	1 & 0 & 1 & 1 & 0 & 1 & 0 & 0 & 1 & 1 & 0 & 1 & 1 & 0 & 1 & 0 \\
	1 & 1 & 0 & 0 & 0 & 0 & 0 & 0 & 0 & 1 & 0 & 1 & 0 & 1 & 0 & 0 \\
	1 & 0 & 1 & 0 & 1 & 1 & 1 & 0 & 0 & 1 & 1 & 1 & 0 & 0 & 0 & 0 \\
	\end{array}
	\right)}.$$
This graph is in $\mathcal{F}_n$, according to the standard decomposition  $d_n=2\times 5^2\times7\times63689\times3118319\times 2740960403$. Now $v=(2, 3, 0, 1, 1, 4, 0, 4, 3, 1, 1, 0, 0, 0, 4, 1)^\T$ is a nontrivial solution to $W^\T z\equiv 0\pmod{5}$. Removing all zero entries, we obtain $v^*=(2, 3, 1, 1, 4, 4, 3, 1, 1, 4, 1)^\T$.
\begin{table}[htbp]
	\renewcommand\arraystretch{1.5}
	\footnotesize
	\centering
	\caption{\label{allp2} Insufficient  perfect $5$-representatives}
	\begin{tabular}{c|c|c|c|c}
		\toprule
		$k$ &shortest $5$-representative $u$&$\frac{1}{5}(u^\T e-5)$&$u^\T u$&perfect $p$-representatives\\
		\midrule
		\multirow{2}{*}{$1$}&
		\multirow{2}{*}{$(2, -2, 1, 1, -1, -1, -2, 1, 1, -1, 1)^\T$}&
		\multirow{2}{*}{$-1$}&\multirow{2}{*}{$20$}& $(2, 3, 1, 1, -1, -1, -2, 1, 1, -1, 1)^\T$ \\
		& & & & $(2, -2, 1, 1, -1, -1, 3, 1, 1, -1, 1)^\T$ \\
			\hline
		
			\multirow{1}{*}{$2$}&
		\multirow{1}{*}{ $(-1, 1, 2, 2, -2, -2, 1, 2, 2, -2, 2)^\T$}&
		\multirow{1}{*}{$0$}&\multirow{1}{*}{$35$}&  \\
		\hline

		\multirow{1}{*}{$3$}&
		\multirow{1}{*}{$(1, -1, -2, -2, 2, 2, -1, -2, -2, 2, -2)^\T$}&
		\multirow{1}{*}{$-2$}&\multirow{1}{*}{$35$}&  \\
		\hline
		\multirow{1}{*}{$4$}&
		\multirow{1}{*}{$(-2, 2, -1, -1, 1, 1, 2, -1, -1, 1, -1)^\T$}&
		\multirow{1}{*}{$-1$}&\multirow{1}{*}{$20$}&  $(3, 2, -1, -1, 1, 1, 2, -1, -1, 1, -1)^\T$ \\
		\bottomrule
	\end{tabular}
\end{table}
Table \ref{allp2} summarizes the execution of Algorithm 1. As the total number of  perfect $p$-representatives is less than the number of nonzero entries in $v$, the graph is DGS.
\end{example}
Table \ref{expt} gives some
experimental results on the DGS-property of graphs with at most 20 vertices. Using Mathematica, for each $n\in \{10,11,\ldots,20\}$, we randomly generate 10,000 graphs. The second column records the number of graphs that are in $\mathcal{F}_n$, while the last column records further the number of graphs which are not DGS, using Algorithm 1.  It seems that the density of $\mathcal{F}_n$ is nearly stable (about 3\%), while the density of non-DGS graphs in $\mathcal{F}_n$ decreases dramatically as $n$ increases.

Haemers \cite{ervdamLAA2003,haemers2016} conjectured that almost all graphs are determined by their spectra. A weaker version of Haemers' conjecture is that almost all graphs are DGS. We note that the
observed phenomenon of the decreasing  density of non-DGS graphs is consistent with the prediction of the weaker version of Haemers' conjecture, and therefore provides some evidences for  it.

\begin{table}[htbp]
	\renewcommand\arraystretch{1.5}
	\centering
	\caption{\label{expt}$\mathcal{F}_n$ and Non-DGS graphs in $\mathcal{F}_n$}
	\begin{tabular}{c|c|c}
		\toprule
		$n$ &\# $\mathcal{F}_n$&\# Non-DGS\\
		\midrule
		$10$ &$278$&52\\
		$11$ &$280$&41\\
		$12$ &$296$&30\\
        $13$ &$323$&22\\
        $14$ &$323$&23\\
        $15$ &$330$&7\\
        $16$ &$344$&3\\
        $17$ &$353$&4\\
        $18$ &$347$&2\\
        $19$ &$300$&0\\
        $20$ &$335$&2\\
		\bottomrule
	\end{tabular}
\end{table}
\section{A conjecture}
In this paper, we have presented an algorithm to check whether a graph $G\in \mathcal{F}_n$ is DGS or not. The key ingredient of the algorithm is to decide whether a vector can generate a primitive matrix. Although this can be done algorithmically, it is still very desirable to  give some more
`evident' conditions either for guaranteeing a vector to generate a primitive matrix, or for  ruling out such a possibility. Motivated by the observed experimental phenomena that the percentage of non-DGS graphs in $\mathcal{F}_n$ has a significant declining trend, we propose the following conjecture for further study.
\begin{conjecture}\label{shortlong}
	Let $v$ be an $n$-dimensional integral vector with each entry nonzero modulo $p$. Suppose that $v^\T e\equiv 0\pmod{p}$ and $v^\T v\equiv 0\pmod{p}$. Then \\
	(\rmnum{1})  If $n\le 8$ then $v$ can always generate some primitive matrix.\\
	(\rmnum{2})  If $n\ge 2p+1$ then $v$ cannot generate any primitive matrix.
\end{conjecture}
We remark that if Conjecture \ref{shortlong} is true, then the final results of Examples \ref{exmgcd} and \ref{exmdgs} can be easily predicted after nontrivial solutions of $W^\T z\equiv 0\pmod{p}$ are found. Indeed, in Example \ref{exmgcd}, the nontrivial solution $v$ of $W^\T z\equiv 0\pmod{5}$ has exactly 8 nonzero entries, which constitutes a vector $v^*$. Noting that $(v^*)^\T e\equiv 0\pmod{p}$ and $(v^*)^\T v^*\equiv 0\pmod{p}$, Conjecture \ref{shortlong} (\rmnum{1}) implies that $v^*$ and hence $v$ can generate a primitive matrix. Nevertheless, in  Example \ref{exmdgs},
the nontrivial solution $v$ has exactly 11 nonzero entries, which  reaches $2p+1$ (noting $p=5$). Thus, we may `predict' that $v$ cannot generate any primitive matrix assuming Conjecture \ref{shortlong} (\rmnum{2}).
\section*{Acknowledgments}
This work is supported by the	National Natural Science Foundation of China (Grant Nos. 12001006, 11971376 and 11971406) and the Scientific Research Foundation of Anhui Polytechnic University (Grant No.\,2019YQQ024).


\begin{thebibliography}{99}
	

  \bibitem{ervdamLAA2003} E.~R.~van~Dam, and W.~H.~Haemers, \emph{Which graphs are determined by their spectrum}? Linear Algebra Appl. \textbf{373} (2003) 241-272.
 \bibitem{ervdamDM2009} E.~R.~van~Dam, and W.~H.~Haemers, \emph{Developments on spectral characterizations of graphs}, Discrete Math. \textbf{309} (2009) 576-586.
 \bibitem{godsil1982}C.~D.~Godsil, and B.~D.~McKay, \emph{Constructing cospectral graphs}, Aequationes Math.\textbf{ 25} (1982) 257-268.
\bibitem{haemers2016} W.~H.~Haemers, \emph{Are almost all graphs determined by their spectrum}? Not. S. Afr. Math. Soc. \textbf{47}(2016) 42-45.
 \bibitem{hj2012}	R.~A.~Horn, and C.~R.~Johnson,	\emph{Matrix Analysis}, 2nd ed., Cambridge University Press, 2012.

\bibitem{ihringer2019LAA} F.~Ihringer, and A.~Munemasa, \emph{New strongly regular graphs from finite geometries via switching}, Linear Algebra Appl. \textbf{580}(2019)464-474.
 \bibitem{ihringer2021DM} F.~Ihringer, F.~Pavese, and V.~Smaldore,  \emph{Graphs cospectral with NU($n+1,q^2$), $n\neq 3$}, Discrete Math. \textbf{344}(2021) 112560.
 \bibitem{johnson1980JCTB} C.~R.~Johnson, and M.~Newman, \emph{A note on cospectral graphs}, J. Combin. Theory, Ser. B \textbf{ 28} (1980) 96-103.

  \bibitem{qiuarXiv} L.~Qiu, W.~Wang, W.~Wang,  and H.~Zhang, Smith normal form and the generalized spectral characterization of graphs, available at http://arxiv.org/abs/2108.00592.

 \bibitem{wang2006EuJC} W. Wang,  and C.-X. Xu, \emph{A sufficient condition for a family of graphs being determined by their generalized spectra}, European J. Combin. \textbf{ 27} (2006) 826-840.
\bibitem{wang2019LAA}  W.~Wang, L.~Qiu, and H.~Yu, \emph{Cospectral graphs, GM-switching and regular rational orthogonal matrices of level $p$}, Linear Algebra Appl.  \textbf{563}(2019)154-177.

 \bibitem{wang2013ElJC} W.~Wang, \emph{Generalized spectral characterization revisited}, Elec. J. Combin. \textbf{ 20} (2013), \#P4.
  \bibitem{wang2017JCTB} W.~Wang, \emph{A simple arithmetric criterion for graphs being determined by their generalized spectra}, J. Combin. Theory, Ser. B  \textbf{122} (2017) 438-451.
 \end{thebibliography}
\end{document}